\documentclass[12pt,reqno]{amsart}

\usepackage[T1]{fontenc}
\usepackage[utf8]{inputenc}
\usepackage{amsmath,amssymb,amsthm,mathtools}
\usepackage{xcolor}
\usepackage{a4wide}
\usepackage{microtype}
\usepackage{enumitem}
\usepackage{needspace}
\usepackage{hyperref}
\newtheorem{theorem}{Theorem}[section]
\newtheorem{proposition}[theorem]{Proposition}

\theoremstyle{definition}

\theoremstyle{remark}

\newcommand{\R}{\mathbb{R}}

\title[Strong Lipschitz Triviality]{On Strong Bi-Lipschitz Triviality of Deformations}
\author{Debomita Chakraborty and Saurabh Trivedi}
\date{}
\address{Indian Institute of Technology Goa, at GEC Campus, Farmagudi, Ponda Goa 403110}
\email{debomita23232201@iitgoa.ac.in}

\address{Indian Institute of Technology Goa, at GEC Campus, Farmagudi, Ponda Goa 403110}
\email{saurabh@iitgoa.ac.in}

\begin{document}
\setlength\parindent{0pt}
\parskip .2cm
\maketitle

\begin{abstract}
The Thom-Levine theorem is a classical method for proving the triviality of deformations by integrating suitable vector fields. In this note, we show that the converse of this criterion does not hold in the bi-Lipschitz category. More precisely, we give an example of a smooth one-parameter deformation of a real function germ in two variables which is bi-Lipschitz trivial but not strongly bi-Lipschitz trivial. The central germ has an isolated critical point, though it is not finitely determined.
\end{abstract}
\section{Introduction}

Let $f : (\R^n,0) \to (\R,0)$ be a smooth function germ. A one-parameter deformation of $f$ is a smooth germ
$$F : (\R \times \R, 0) \to (\R,0)$$
such that $F(x,0) = f(x).$ Writing $f_t(x) = F(x,t)$, we also assume $F(0,t) = 0$. That is, every $f_t$ is a function germ at the origin.

The deformation $F$ is said to be bi-Lipschitz trivial if there exists a bi-Lipschitz homeomorphism germ $$H : (\R^n \times \R,0)\to (\R^n \times \R,0),$$
given by $H(x,t) = (h_t(x),t)$, $h_0(0) = 0$, $h_0 = id$, such that 
$F(h_t(x),t)=f(x).$
Thus, all germs $f_t$ are simultaneously bi-Lipschitz right-equivalent to $f$. 

The Lipschitz version of the Thom-Levine theorem, see for example \cite{Levine1971,Martinet1982}, states the following:

Suppose there exists a vector field
$$X = \frac{\partial}{\partial t} + \sum_{i=1}^n v_i(x,t) \frac{\partial}{\partial x_i},$$
where the $v_i$ are locally Lipschitz in $x$, uniformly with respect to $t$, continuous in $t$, and satisfy $v_i(0,t) = 0$, such that 
$$X.F = \frac{\partial F}{\partial t} + \sum_{i=1}^n v_i(x,t) \frac{\partial F}{\partial x_i} = 0.$$
Then, $F$ is bi-Lipschitz trivial. 

In fact, the Lipschitz vector field can be integrated to get a unique local flow $h_t$, and
$\frac{d}{dt} F(h_t(x),t) = V.F(h_t(x),t) = 0.$
Hence
$$F(h_t(x),t) = F(x,0) = f(x).$$
The bi-Lipschitz trivialization obtained by integrating such a vector field is called strong bi-Lipschitz triviality. 

This Lipschitz analogue of the Thom-Levine Theorem has become a useful tool for proving bi-Lipschitz triviality. In particular, this method was used by Fernandes and Ruas to establish the bi-Lipschitz determinacy of quasihomogeneous function germs \cite{FernandesRuas2004}, by Costa, Saia and Soares Júnior to obtain bi-Lipschitz ($\mathcal A$)-triviality results using Newton filtrations \cite{CostaSaiaSoares2012}, and by Nguyen, Ruas and Trivedi in the classification of Lipschitz simple and Lipschitz unimodal function germs \cite{NguyenRuasTrivedi2020,NRT2026}; see also \cite{NguyenTrivedi2022} for a detailed account of this method.

In this article, we will show that the converse of the Thom–Levine theorem does not hold in general, that is the strong bi-Lipschitz triviality of a deformation does not imply strong bi-Lipschitz triviality. We will present a smooth deformation of a smooth function constructed using a flat function that is bi-Lipschitz trivial, but for which there is no Lipschitz vector field whose flow provides a trivialization. The proposed counterexample is, in fact, a deformation of an isolated singularity; however, it is not finitely determined. It remains unknown whether the converse holds for finitely determined germs.

\section{The counterexample}

Let
$$\rho(u) = \begin{cases}
    e^{-\frac{1}{u^2}},& \qquad u \neq 0\\
    0, &\qquad u = 0.    
\end{cases}$$
Let $0 < \varepsilon < 1,$ and define
$$
\kappa(x,t)=
\begin{cases}
\displaystyle x+\varepsilon\frac{tx}{\sqrt{x^{2}+t^{2}}},
& (x,t)\neq(0,0),\\
0, & (x,t)=(0,0),
\end{cases}
$$

Consider the deformation $F : (\R^2 \times \R,0) \to (\R,0)$ defined by
$$F(x,y,t) = \rho(\kappa(x,t)) + y^2.$$
Then, $F(x,y,0) = \rho(\kappa(x,0)) + y^2 = \rho(x) + y^2 = f_0$.

It is easy to see that for every fixed $t$, $F(x,y,t)$ has an isolated singularity at $(0,0)$. However, it is not finitely-determined.

\begin{proposition}
The deformation $F$ is a smooth.
\end{proposition}

\begin{proof}
Note that $\kappa$ is smooth away from $(0,0)$. Therefore, $G(x,t)=\rho(\kappa(x,t))$, being a composition of two smooth functions, is smooth everywhere except possibly at $(0,0)$. Moreover $\kappa$ is continuous at $(0,0)$. This can be seen as follows:
Since, $$|\kappa(x,t) \leq |x| + \varepsilon \frac{|tx|}{\sqrt{x^2+t^2}} \leq (1+\varepsilon)|x|,$$
$\kappa(x,t) \to 0 = \kappa(0,0)$ as $(x,t) \to (0,0).$ 

However, $\kappa$ is not differentiable at $(0,0)$. This can be seen as follows:

We will take limits approaching $(0,0)$ along the line $x = t$. Note that,
$$\kappa(h,h) = h + \varepsilon \frac{h^2}{\sqrt{2h^2}} = h+ \frac{\varepsilon}{\sqrt{2}}|h|.$$
Since $\kappa(0,0) = 0,$
$$\frac{\kappa(h,h) - \kappa(0,0)}{h} = 1 + \frac{\varepsilon}{\sqrt{2}}\frac{|h|}{h}.$$
Therefore,
$$\lim_{h\to 0^+} \frac{\kappa(h,h)}{h} = 1 + \frac{\varepsilon}{\sqrt{2}}, \qquad \lim_{h\to 0^-} \frac{\kappa(h,h)}{h} = 1 - \frac{\varepsilon}{\sqrt{2}}.$$
The limits differ when $\varepsilon \neq 0$. Thus, the directional derivative along $(1,1)$ does not exist, so $\kappa$ is not differentiable at $(0,0)$.

We will show, nevertheless, that $G$ is smooth everywhere. The only point where we need to test the smoothness is at the origin.

Put $z = (x,t)$ and $r = \|z\|.$ The function $\kappa$ is smooth on $\R^2\setminus\{0\}$ and it is also positively homogeneous of degree $1$, that is, $\kappa(\lambda z) = \lambda \kappa(z)$ for any $\lambda > 0$. Therefore, for every multi-index $\beta=(\beta_1,\beta_2)$ with $|\beta|=\beta_1+\beta_2\geq 1$, the derivative $D^\beta\kappa$ is positively homogeneous of degree $1-|\beta|$. Here

$$
D^\beta=\frac{\partial^{|\beta|}}{\partial x^{\beta_1}\partial t^{\beta_2}}.
$$

Indeed, differentiating the identity $\kappa(\lambda z)=\lambda\kappa(z)$ gives

$$
D^\beta\kappa(\lambda z)=\lambda^{1-|\beta|}D^\beta\kappa(z).
$$

Writing $z=r\omega$, where $r=\|z\|$ and $|\omega|=1$, we obtain

$$
D^\beta\kappa(z)=r^{1-|\beta|}D^\beta\kappa(\omega).
$$

Since $D^\beta\kappa$ is continuous on the unit circle, it is bounded there. Hence, for every $\beta$, there is a constant $C_\beta>0$ such that

\[
|D^\beta\kappa(z)|\leq C_\beta r^{1-|\beta|}. \tag{1}
\]

We now use the flatness of $\rho$. Recall that

$$
\rho(u)=
\begin{cases}
e^{-1/u^2},&u\neq 0,\\
0,&u=0.
\end{cases}
$$

For every integer $j\geq 0$, the derivative $\rho^{(j)}(u)$, for $u\neq 0$, is of the form

$$
\rho^{(j)}(u)=P_j(1/u)e^{-1/u^2},
$$

where $P_j$ is a polynomial. Since $e^{-1/u^2}$ tends to zero faster than every power of $|u|$, for every pair of non-negative integers $j,N$ there is a constant $C_{j,N}>0$ such that

\[
|\rho^{(j)}(u)|\leq C_{j,N}|u|^N
\tag{2}
\]

for $u$ sufficiently close to $0$.

Also, from the estimate obtained above,

\[
|\kappa(x,t)|\leq (1+\varepsilon)|x|\leq (1+\varepsilon)r.
\tag{3}
\]

Let now $\alpha$ be a multi-index with $|\alpha|=m$. On $\R^2\setminus{0}$, repeated application of the chain rule and the product rule shows that every term appearing in $D^\alpha G$ is of the form

\[
\rho^{(j)}(\kappa(z))
D^{\beta_1}\kappa(z)\cdots D^{\beta_j}\kappa(z),
\tag{4}
\]

where each $\beta_i$ is a non-zero multi-index and

$$
|\beta_1|+\cdots+|\beta_j|=m.
$$

Using (1), we have

$$
\left|D^{\beta_1}\kappa(z)\cdots D^{\beta_j}\kappa(z)\right|
\leq C r^{j-m}.
$$

On the other hand, by (2) and (3), for every $N\geq 0$,

$$
|\rho^{(j)}(\kappa(z))|
\leq C|\kappa(z)|^N
\leq Cr^N.
$$

Consequently, every term in (4) is bounded by

$$
Cr^{N+j-m}.
$$

Since $N$ can be chosen arbitrarily large, it follows that, for every multi-index $\alpha$ and every integer $L\geq 0$, there is a constant $C_{\alpha,L}>0$ such that

\[
|D^\alpha G(z)|\leq C_{\alpha,L}r^L
\tag{5}
\]

for $z\neq 0$ sufficiently close to the origin.

We now show that all derivatives of $G$ exist at the origin and vanish there. We proceed by induction on their order. For order zero, (5) gives $G(z)\to 0=G(0)$, so $G$ is continuous at the origin.

Suppose that all derivatives $D^\alpha G$ of order $m$ exist at the origin and satisfy $D^\alpha G(0)=0$. Then

$$
\frac{\partial}{\partial x}D^\alpha G(0)
=
\lim_{h\to 0}
\frac{D^\alpha G(h,0)-D^\alpha G(0)}{h}.
$$

Taking $L=2$ in (5), we have $|D^\alpha G(h,0)|\leq C|h|^2$. Therefore,

$$
\left|
\frac{D^\alpha G(h,0)}{h}
\right|
\leq C|h|\longrightarrow 0.
$$

Thus the derivative with respect to $x$ exists at the origin and is zero. The same argument, using the points $(0,h)$, proves the same for the derivative with respect to $t$. Hence all derivatives of order $m+1$ exist at the origin and vanish there.

Finally, taking $L=1$ in (5), we obtain
$$
|D^\alpha G(z)|\leq C_{\alpha,1}r\longrightarrow 0=D^\alpha G(0).
$$
Thus every partial derivative of $G$ is continuous at the origin. Hence $G$ is smooth at $(0,0)$, and therefore smooth everywhere.

Since
$$
F(x,y,t)=G(x,t)+y^2,
$$
and $y^2$ is smooth, the deformation $F$ is smooth.

\end{proof}

We will now show that $F$ is bi-Lipschitz trivial.
\begin{proposition}
The deformation
\[
F(x,y,t)=\rho(\kappa(x,t))+y^{2}
\]
is bi-Lipschitz trivial.
\end{proposition}

\begin{proof}
Define the parameter-preserving map
\[
\Phi(x,y,t)=\bigl(\kappa(x,t),y,t\bigr).
\]
We first prove that \(\Phi\) is bi-Lipschitz.

Away from \((x,t)=(0,0)\), direct differentiation gives
\[
\frac{\partial\kappa}{\partial x}
=
1+\varepsilon
\frac{t^{3}}{(x^{2}+t^{2})^{3/2}},
\qquad
\frac{\partial\kappa}{\partial t}
=
\varepsilon
\frac{x^{3}}{(x^{2}+t^{2})^{3/2}}.
\]
Set \(e=|\varepsilon|<1\). Since
\[
\left|
\frac{t^{3}}{(x^{2}+t^{2})^{3/2}}
\right|
\leq 1,
\qquad
\left|
\frac{x^{3}}{(x^{2}+t^{2})^{3/2}}
\right|
\leq 1,
\]
we obtain
\[
1-e
\leq
\frac{\partial\kappa}{\partial x}
\leq
1+e,
\qquad
\left|
\frac{\partial\kappa}{\partial t}
\right|
\leq e.
\]

Consequently, for every fixed \(t\),
\[
(1-e)|x_{1}-x_{2}|
\leq
|\kappa(x_{1},t)-\kappa(x_{2},t)|
\leq
(1+e)|x_{1}-x_{2}|.
\tag{1}
\]
When \(t=0\), this follows from \(\kappa(x,0)=x\); when
\(t\neq0\), it follows from the mean value theorem. Similarly,
for every fixed \(x\),
\[
|\kappa(x,t_{1})-\kappa(x,t_{2})|
\leq
e|t_{1}-t_{2}|.
\tag{2}
\]
For \(x=0\), this is immediate because \(\kappa(0,t)=0\); for
\(x\neq0\), it follows from the mean value theorem.

Combining \((1)\) and \((2)\), we obtain
\begin{align*}
|\kappa(x_{1},t_{1})-\kappa(x_{2},t_{2})|
&\leq
|\kappa(x_{1},t_{1})-\kappa(x_{2},t_{1})| \\
&\quad+
|\kappa(x_{2},t_{1})-\kappa(x_{2},t_{2})| \\
&\leq
(1+e)|x_{1}-x_{2}|+e|t_{1}-t_{2}|.
\end{align*}
Therefore the map
\[
(x,t)\longmapsto \bigl(\kappa(x,t),t\bigr)
\]
is Lipschitz.

For each fixed \(t\), inequality \((1)\) shows that
\(x\mapsto\kappa(x,t)\) is strictly increasing. Moreover,
\[
|\kappa(x,t)|\geq(1-e)|x|,
\]
and therefore \(\kappa(x,t)\to\pm\infty\) as \(x\to\pm\infty\).
Thus \(x\mapsto\kappa(x,t)\) is a homeomorphism of
\(\mathbb{R}\). Denote its inverse by
\[
x=\psi(u,t).
\]

To prove that the inverse depends Lipschitz continuously on both
\(u\) and \(t\), let
\[
u_i=\kappa(x_i,t_i), \qquad i=1,2.
\]
Using \((1)\) and \((2)\), we obtain
\begin{align*}
(1-e)|x_{1}-x_{2}|
&\leq
|\kappa(x_{1},t_{1})-\kappa(x_{2},t_{1})| \\
&\leq
|\kappa(x_{1},t_{1})-\kappa(x_{2},t_{2})| \\
&\quad+
|\kappa(x_{2},t_{2})-\kappa(x_{2},t_{1})| \\
&\leq
|u_{1}-u_{2}|+e|t_{1}-t_{2}|.
\end{align*}
Hence
\[
|x_{1}-x_{2}|
\leq
\frac{|u_{1}-u_{2}|+e|t_{1}-t_{2}|}{1-e}.
\tag{3}
\]
It follows that
\[
(u,t)\longmapsto\bigl(\psi(u,t),t\bigr)
\]
is Lipschitz. Thus
\[
(x,t)\longmapsto\bigl(\kappa(x,t),t\bigr)
\]
is bi-Lipschitz.

Since \(\Phi\) leaves the \(y\)-coordinate unchanged, it follows
that
$
\Phi(x,y,t)=\bigl(\kappa(x,t),y,t\bigr)
$
is bi-Lipschitz, with inverse
$
\Phi^{-1}(u,v,t)=\bigl(\psi(u,t),v,t\bigr).
$

Finally,
\begin{align*}
F\bigl(\Phi^{-1}(u,v,t)\bigr)
&=
F\bigl(\psi(u,t),v,t\bigr) \\
&=
\rho\bigl(\kappa(\psi(u,t),t)\bigr)+v^{2} \\
&=
\rho(u)+v^{2} \\
&=
f(u,v).
\end{align*}
Therefore \(\Phi\) is a parameter-preserving bi-Lipschitz
trivialization of \(F\). Equivalently, if
$
\phi_t(x,y)=\bigl(\kappa(x,t),y\bigr),
$
then
$
F_t=f\circ\phi_t.$
Hence $F$ is bi-Lipschitz right trivial.
\end{proof}

Finally, we show that the deformation in not strongly bi-Lipschitz trivial.

\begin{proposition}
The deformation \(F\) is not strongly bi-Lipschitz trivial.
\end{proposition}

\begin{proof}
Suppose, on the contrary, that \(F\) is strongly bi-Lipschitz right
trivial. Then there exist Lipschitz function germs \(a\) and \(b\)
satisfying the equation
\(F_t+aF_x+bF_y=0\).

Consider the points \((x,0,0)\), where \(x\neq0\). Since \(\kappa\)
depends only on \(x\) and \(t\), we have
\[
\kappa(x,0)=x,\qquad
\kappa_x(x,0)=1,\qquad
\kappa_t(x,0)=\varepsilon\operatorname{sgn}(x).
\]

Since \(F=\rho(\kappa)+y^2\), it follows that
\(F_x(x,0,0)=\rho'(x)\), \(F_y(x,0,0)=0\), and
\(F_t(x,0,0)=\varepsilon\operatorname{sgn}(x)\rho'(x)\).
Therefore the Thom--Levine equation gives
\[
\rho'(x)\bigl(\varepsilon\operatorname{sgn}(x)
+a(x,0,0)\bigr)=0.
\]
Now \(\rho'(x)=2x^{-3}e^{-1/x^2}\neq0\) for \(x\neq0\). Hence
\(a(x,0,0)=-\varepsilon\operatorname{sgn}(x)\). Consequently,
\(\lim_{x\to0^+}a(x,0,0)=-\varepsilon\), whereas
\(\lim_{x\to0^-}a(x,0,0)=\varepsilon\). Since
\(\varepsilon\neq0\), the function \(a\) is not continuous at the
origin, contradicting its Lipschitz continuity. Therefore \(F\) is
not strongly bi-Lipschitz right trivial.
\end{proof}


\begin{thebibliography}{99}

\bibitem{Levine1971}
H.~I. Levine,
\emph{Singularities of differentiable mappings},
in: Proceedings of Liverpool Singularities Symposium I,
Lecture Notes in Mathematics, vol.~192,
Springer-Verlag, Berlin, 1971, pp.~1--21.

\bibitem{Martinet1982}
J. Martinet,
\emph{Singularities of Smooth Functions and Maps},
London Mathematical Society Lecture Note Series, vol.~58,
Cambridge University Press, Cambridge, 1982.

\bibitem{FernandesRuas2004}
A.~C.~G. Fernandes and M.~A.~S. Ruas,
Bi-Lipschitz determinacy of quasihomogeneous germs,
\emph{Glasgow Math. J.} \textbf{46} (2004), no.~1, 77--82.

\bibitem{CostaSaiaSoares2012}
J.~C.~F. Costa, M.~J. Saia and C.~H. Soares J\'unior,
Bi-Lipschitz \(\mathcal A\)-triviality of map germs and Newton
filtrations,
\emph{Topology Appl.} \textbf{159} (2012), no.~2, 430--436.

\bibitem{NguyenRuasTrivedi2020}
N. Nguyen, M.~A.~S. Ruas and S. Trivedi,
Classification of Lipschitz simple function germs,
\emph{Proc. Lond. Math. Soc.} (3) \textbf{121} (2020),
no.~1, 51--82.

\bibitem{NRT2026}
N. Nguyen, M.~A.~S. Ruas and S. Trivedi,
Classification of Lipschitz Unimodal function germs,
\emph{Preprint} (2026).

\bibitem{NguyenTrivedi2022}
N. Nguyen and S. Trivedi,
Invariants and classification of simple function germs with respect
to Lipschitz \(\mathcal A\)-equivalence,
\emph{J. Singul.} \textbf{25} (2022), 348--360.

\end{thebibliography}
\end{document}